\documentclass[11pt]{amsart}

\usepackage{amsmath}
\usepackage{amssymb}
\usepackage{graphicx}
\usepackage[colorlinks=true,linkcolor=blue,citecolor=blue]{hyperref}
\usepackage{csquotes}
\usepackage{amsrefs}

\newtheorem{theorem}{Theorem}[section]

\newtheorem{lemma}[theorem]{Lemma}

\theoremstyle{definition}
\newtheorem{definition}[theorem]{Definition}
\newtheorem{example}[theorem]{Example}
\newtheorem{remark}[theorem]{Remark}

\numberwithin{equation}{section}
\title[Well-posedness for the split equilibrium problem]{Well-posedness for the split equilibrium problem}

\author[S. Dey]{Soumitra Dey$^*$}\thanks{$^*$Correspondence authors.}
\address[S. Dey]{Department of Mathematics, India Institute of Technology Madras, Chennai, India}
\email{\tt deysoumitra2012@gmail.com}

\address[S. Dey]{ Department of Mathematics, The Technion--Israel Institute of Technology, Haifa, 3200003, Israel}

\author{V. Vetrivel}
\address[V. Vetrivel]{Department of Mathematics, India Institute of Technology Madras, Chennai, India}
\email{\tt vetri@iitm.ac.in}

\author[H. K. Xu]{Hong-Kun Xu$^*$}
\address[H. K. Xu]{Department of Mathematics, School of Science, Hangzhou Dianzi University,
     Hangzhou, 310018, China}
\email{\tt xuhk@hdu.edu.cn}

\address[H. K. Xu]{College of Mathematics and Information Science,  Henan Normal University,
     Xinxiang, 453007, China}



\begin{document}

\begin{abstract}
We extend the concept of well-posedness to the split equilibrium problem and establish Furi-Vignoli-type characterizations for the well-posedness. We prove that the well-posedness of the split equilibrium problem is equivalent to the existence and uniqueness of its solution under certain assumptions on the bifunctions involved. We also characterize the generalized well-posedness of the split equilibrium problem via the Kuratowski measure of noncompactness. We illustrate our theoretical results by several examples.
\end{abstract}

\maketitle

%======================================================================================================
\section{Introduction}
\label{Sec:1}
The concept of well-posedness is one of the most important and interesting subjects in nonlinear analysis to study many optimization problems such as variational inequalities (VIs), split VIs, equilibrium problems, inverse problems. The notion of well-posedness for a minimization problem was first introduced and widely studied by Tykhonov \cite{ANTY1966} in a metric space. A minimization problem is said to be well-posed (or correctly posed) if it has a unique optimal solution, and every minimizing sequence converges to the unique solution (see, for instance, \cite{MFUR1970}). Since the requirement of the uniqueness of the solution may be too restrictive when dealing with applications to some optimal controls or  mathematical programming problems, the above definition of well-posedness was relaxed by Furi and Vignoli \cite{MFUR1970} for unconstrained minimization problems. Lucchetti and Patrone in \cite{RLUC1981, RLUC1982} extended the well-posedness concept to VIs. It is important to mention that their well-posedness is the first notion of well-posedness for VIs. Later on, the concept of well-posedness was extended by many authors (see, \cite{MBLI2000,TZOL1995,TZOL1996,MBLI2006,YPFA2008,YPFA2010,XXHU2009,CSLA2012,XBLI2013,MMWO2013,YBXI2011}).

Let $C$ be a nonempty closed convex subset of a real Hilbert space $H$ and let $g:H\rightarrow H$ be a mapping. Then VI is to find a point $x^*\in C$ with the following property:
\begin{align}\label{VI}
\left\langle g(x^*), p-x^*\right\rangle\geq 0, \quad\forall p\in C.
\end{align}

VI is an useful tool for solving many problems such as systems of linear and nonlinear equations, complementarity problems, and saddle point problems (see, for more applications, \cite{FFAC2003}). VI was first introduced and studied by Stampacchia \cite{GSTA1964} in the finite-dimensional Euclidean space. In 2003, Noor \cite{MDNO2003} introduced a concept of well-posedness for a new class of VIs and established some results under the pseudo-monotonicity assumption on the mapping involved. Thereafter many researchers have put their attention to VI (\ref{VI}) by generalizing in various directions (see, for instance, \cite{FFAC2003, QLDO2017}) for both finite- and infinite-dimensional settings.

One of the most important generalizations to VI is the split variational inequality (SVI). Let $H_1$ and $H_2$ be two real Hilbert spaces and let $C$ and $Q$ be two non-empty closed convex subsets of $H_1$ and $H_2$, respectively. The SVI is  to find a point $x^*\in C$ with the properties
\begin{align}\label{SVIP}
\left\langle f(x^*), p-x^*\right\rangle\geq 0, \quad\forall p\in C
\end{align}
and
\begin{align}
y^*=Ax^*\in Q\quad\text{solves}\quad\left\langle g(y^*), q-y^*\right\rangle\geq 0, \quad\forall q\in Q,\nonumber
\end{align}
where $f:H_1\rightarrow H_1$ and $g:H_2\rightarrow H_2$ are two mappings and $A:H_1\rightarrow H_2$ is a bounded linear operator. For more details on SVI, one is referred to \cite{YCEN2010, QHAN2014} and the references therein. Recently, Hu, et. al. \cite{RHUA2014} introduced and studied the well-posedness of the SVI (\ref{SVIP}) in reflexive Banach spaces. 

Another important generalization to VI is known as equilibrium problem (EP).
Let $C$ be a non-empty closed convex subset of a real Hilbert space $H$. Let $f:C\times C\rightarrow\mathbb{R}$ be a bifunction.
Then the EP consists of finding an $x^*\in C$ such that
\begin{align}\label{EP}
f(x^*, p)\geq 0,\quad\forall p\in C.
\end{align}

EPs have received the attention of many researchers in the recent years %(see for instance \cite{GKAS2000}, \cite{MBIA2005}-\cite{ANIU2003}) and the references therein,
due to their important special cases such as optimization problems \cite{TZOL1996}, saddle point problems \cite{MBIA2010}, VIs \cite{GSTA1964, EBLU1994} and Nash EPs \cite{FFAC2003}, minimax problems. 
These problems are useful models of many practical problems arising in game theory, physics, economics, etc. Another important part of EPs is that they consider all the special cases in a unified form. Many extensions of EPs can be found \cite{LESL2019,BALL2018,MBIA2007} and the references therein.

Motivated by the above-mentioned well-posedness of VIs and other related problems,  Bianchi, et. al. \cite{MBIA2010} introduced some concepts of well-posedness for scalar EPs in complete metric spaces or in Banach spaces.

Similarly, another important generalization to EP $(\ref{EP})$ is known as the split equilibrium problem (SEP). Let $H_1$ and $H_2$ be two real Hilbert spaces and let $C$ and $Q$ be two non-empty closed convex subsets of $H_1$ and $H_2$, respectively. An SEP is to find an $x^*\in C$ such that
\begin{align}\label{SEP}
 f(x^*, p)\geq 0, \quad\forall p\in C
\end{align}
and
\begin{align}
y^*=Ax^*\in Q\quad\text{solves}\quad g(y^*, q)\geq 0, \quad\forall q\in Q,\nonumber
\end{align}
where $f:C\times C\rightarrow\mathbb{R}$ and $g:Q\times Q\rightarrow\mathbb{R}$ are two bifunctions, and $A:H_1\rightarrow H_2$ is a bounded linear operator. It is easy to show that under certain constraint qualifications SEP reduces to SVI. 

Motivated and inspired by the concept of well-posedness to variational problems in recent years, in this paper we shall introduce the concept of the well-posedness to the SEP (\ref{SEP}).

This paper is organized as follows. In section \ref{Sec:2} we collect some basic definitions and results. In section \ref{Sec:3} we introduce the concepts of an approximating sequence for the SEP (\ref{SEP}) and we introduce various kinds of well-posedness for the SEP (\ref{SEP}). Also, we establish the metric characterization of well-posedness and generalized well-posedness for the SEP (\ref{SEP}). In this section we have also provided some nontrivial examples to illustrate our theoretical analysis. In section \ref{Sec:4} we prove that the well-posedness of the SEP is equivalent to the existence and uniqueness of its solution. In section \ref{Sec:5} we draw a conclusion of our work.

%=======================================================================================================
\section{Preliminaries}
\label{Sec:2}
In this section we collect some basic definitions and results which we will use in our analysis. We refer \cite{KKUR1968, EKLE1984,HHBA2011,VRAK1998,EBLU1994} for more details.

Let $A$ and $B$ be two nonempty subsets of a real Hilbert space $H$. The Hausdorff metric $H(\cdot, \cdot)$ between $A$ and $B$ is defined by
\begin{align*}
H(A, B)=\max\left\lbrace D(A, B), D(B, A)\right\rbrace,
\end{align*}
where $D(A, B)= \sup_{a\in A}d(a, B)$ with $d(a, B)=\inf_{b\in B}\|a-b\|$.
Let $\left\lbrace A_n\right\rbrace$ be a sequence of subsets of $H.$ We say that $A_n$ converges to $A$ if and only if $H(A_n, A)\rightarrow 0.$ It is easy to see that $D(A_n, A)\rightarrow 0$ if and only if $d(a_n, A)\rightarrow 0$ uniformly for all selection $a_n\in A_n$.

The diameter of a set $B$ is denoted by $\text{diam}(B)$ and is defined by
\begin{align*}
\text{diam}(B)=\sup \left\lbrace \|x-y\|: x, y\in B\right\rbrace.
\end{align*}

\begin{definition}
Let $C$ be a nonempty subset of a real Hilbert space $H$. A function $f:C\rightarrow\mathbb{R}$ is said to be {\em lower semi-continuous} (l.s.c) at a point $x$ in $C$ if for every sequence $\left\lbrace x_n\right\rbrace$ in $C$ with $x_n\rightarrow x$ it follows that
\begin{align*}
f(x)\leq\liminf_{n\rightarrow\infty} f(x_n).
\end{align*}
\end{definition}

\begin{definition}
Let $C$ be a nonempty subset of a real Hilbert space $H$. A function $f:C\rightarrow\mathbb{R}$ is said to be {\em upper semi-continuous} (u.s.c) at a point $x$ in $C$ if for every sequence $\left\lbrace x_n\right\rbrace$ in $C$ with $x_n\rightarrow x$ it follows that
\begin{align*}
\limsup_{n\rightarrow\infty} f(x_n)\leq f(x).
\end{align*}
\end{definition}

\begin{definition}
Let $G$ be a nonempty subset of a real Hilbert space $H$. The {\em Kuratowski measure of noncompactness} of the set $G$ is defined by
\begin{align}
\alpha(G)=\inf\left\lbrace\epsilon>0: G\subset\cup_{i=1}^n G_i,  \text{diam}(G_i)<\epsilon,~ i=1,2,\cdots, n,~
n\in\mathbb{N} \right\rbrace\nonumber.
\end{align}
\end{definition}

\begin{lemma}
Let $P$ and $Q$ be two nonempty subsets of a real Hilbert space $H$.  Then
\begin{align*}
\alpha(P)\leq 2 H(P, Q)+\alpha(Q),
\end{align*}
where $\alpha$ is the Kuratowski measure of noncompactness.
\end{lemma}

\begin{definition}
Let $K$ be a nonempty subset of real Hilbert space $H$. A bifunction $f:K\times K\rightarrow\mathbb{R}$ is said to be {\em monotone} if
\begin{align}
f(x, y)+f(y, x)\leq 0,\quad\forall x, y\in K.\nonumber
\end{align}
\end{definition}

\begin{definition}
Let $K$ be a nonempty convex subset of real Hilbert space $H$. A bifunction $f:K\times K\rightarrow\mathbb{R}$ is said to be {\em hemicontinuous} if
\begin{align}
\limsup_{t\rightarrow 0^+}f(x+t(y-x), y)\leq f(x, y),\quad\forall x, y\in K.\nonumber
\end{align}
\end{definition}

\begin{theorem}\label{Lem:1}
Let $K$ be convex and let $h: K\times K\rightarrow\mathbb{R}$ be a monotone and hemicontinuous bifunction. Assume that\\
$(i)~h(x, x)\geq 0$ for all $x\in K$.\\
$(ii)$ for every $x\in K$, $h(x,\cdot)$ is convex.\\ Then for given $x^*\in K,$
\begin{align*}
h(x^*, y)\geq 0, \quad\forall y\in K
\end{align*}
if and only if
\begin{align*}
h(y, x^*)\leq 0, \quad\forall y\in K.
\end{align*}
\end{theorem}

%=====================================================================================================
\section{Approximate Sequences, Metric Characterizations and Well-Posedness}
\label{Sec:3}

In this section, we introduce an approximate sequence to the SEP and extend the well-posedness notions to the SEP.
We also derive metric characterizations of the well-posedness.
We always assume that $C$ and $Q$ are two nonempty closed convex subsets of real Hilbert spaces $H_1$ and $H_2$, respectively,
and $A: H_1\to H_2$ is a bounded linear operator.

\begin{definition}
A sequence $\left\lbrace (x_n, y_n)\right\rbrace\in H_1\times H_2$ is said to be an {\em approximate sequence} for the SEP $(\ref{SEP})$ if there exists $0<\epsilon_n\rightarrow 0$ such that
\begin{equation}\label{ASEP}
    \begin{cases}
      x_n\in C, \quad y_n\in Q,\\
      \|y_n-Ax_n\|\leq \epsilon_n, \\
      f(x_n, p)\geq -\epsilon_n, \quad\forall  p\in C,\\
      g(y_n, q)\geq -\epsilon_n, \quad\forall q\in Q.
    \end{cases}
\end{equation}
\end{definition}

\begin{definition}
Let $S$ denote the solution set of the SEP (\ref{SEP}). Then we say that the SEP \eqref{SEP} is well-posed if $S$ is a singleton set and every approximating sequence for the SEP \eqref{SEP} converges to the unique solution. We say that the SEP (\ref{SEP}) is generalized well-posed if $S\neq\emptyset$ and every approximating sequence for the SEP \eqref{SEP} has a subsequence which converges to some element of S.
\end{definition}

\noindent Given $\epsilon\geq 0$, define the following set:
\begin{multline}
S(\epsilon) =  \left\{  (x, y)\in C\times Q: \|y-Ax\|\leq \epsilon; \right.\nonumber\\
 \left. {f(x, p)\geq -\epsilon,\quad\forall  p\in C;~ g(y, q)\geq -\epsilon,\quad\forall q\in Q}\right\}.
\end{multline}
Note that $S(\epsilon)$ is called the approximate solution set to the SEP $(\ref{SEP})$.
Note also that $S(\epsilon)\subset S(\epsilon')$ if $0\le\epsilon<\epsilon'$.

\begin{theorem}\label{Thm:1}
The SEP (\ref{SEP}) is well-posed if and only if its solution set $S$ is nonempty and $\text{diam}(S(\epsilon))\rightarrow 0$
as $\epsilon\rightarrow 0$.
\end{theorem}

\begin{proof}
Suppose the SEP (\ref{SEP}) is well-posed; thus the solution set $S\neq\emptyset$ by definition. We prove
$\text{diam}(S(\epsilon))\rightarrow 0$ (as $\epsilon\rightarrow 0$) by contradiction.
Suppose that $\text{diam}(S(\epsilon))\nrightarrow 0$ (as $\epsilon\rightarrow 0$). Then there exist $\delta>0$, $0<\epsilon_n\rightarrow 0$,
$(x_n, y_n)\in S(\epsilon_n)$ and $(x_n^\prime, y_n^\prime)\in S(\epsilon_n)$ such that
\begin{align}\label{in1}
\|(x_n, y_n)-(x_n^\prime, y_n^\prime)\|>\delta
\end{align}
for all $n\in\mathbb{N}$.
Since the SEP (\ref{SEP}) is well-posed, $(x_n, y_n)$ and $(x_n^\prime, y_n^\prime)$ both converge to the same unique solution. Thus we have
\begin{align*}
\|(x_n, y_n)-(x_n^\prime, y_n^\prime)\|\rightarrow 0\quad\text{as}\quad n\rightarrow\infty.
\end{align*}
This is a contradiction to $(\ref{in1})$.

Conversely, assume that the solution set $S$ of the SEP (\ref{SEP}) is nonempty and $\text{diam}(S(\epsilon))\rightarrow 0$
as $\epsilon\rightarrow 0$. Since $\emptyset\ne S\subset S(\epsilon)$ for all $\epsilon>0$ and $\text{diam}(S(\epsilon))\rightarrow 0$
as $\epsilon\rightarrow 0$, $S$ must be a singleton set.

Let $\left\lbrace (x_n, y_n)\right\rbrace$ be an approximating sequence. Then there exists
$0<\epsilon_n\rightarrow 0$ such that \eqref{ASEP} holds. Therefore, $(x_n,y_n)\in S(\epsilon_n)$ for all $n\in\mathbb{N}$.
Let $(x^*, y^*)$ be the unique solution to the SEP; thus, $(x^*, y^*)\in S(\epsilon_n)$.
It turns out that
\begin{align*}
\|(x_n, y_n)-(x^*, y^*)\|\leq \text{diam}(S(\epsilon_n))\rightarrow 0.
\end{align*}
This proves that the SEP (\ref{SEP}) is well-posed, and the proof is complete.
\end{proof}

\begin{remark}
The above Theorem \ref{Thm:1} provides the equivalence between the well-posedness of the SEP (\ref{SEP}) and properties of its solution and approximate solution sets.
\end{remark}

Let us consider the following example.
\begin{example}
Let $H_1=H_2=\mathbb{R}$, $A=I$ and $C=Q=\mathbb{R}$.
Let us define $f:C\times C\rightarrow\mathbb{R}$ by
\begin{equation*}
f(x, p)=p^2-x^2
\end{equation*}
and $g:Q\times Q\rightarrow\mathbb{R}$ by
\begin{align*}
g(y, q)=-y^2e^{-q^2}.
\end{align*}

It is clear that $S=\{(0,0)\}\subset C\times Q$ is a unique solution. 

Now, for any $\epsilon>0$, 
\begin{align*}
S(\epsilon)&=\{(x,y)\in\mathbb{R}^2:|y-x|\leq\epsilon;p^2-x^2\geq-\epsilon,\forall p\in\mathbb{R};-y^2e^{-q^2}\geq-\epsilon,\forall q\in\mathbb{R}\}\\&
=\{(x,y)\in\mathbb{R}^2:|y-x|\leq\epsilon;|x|\leq\sqrt{\epsilon};|y|\leq\sqrt{\epsilon}\}\\&
\subset [-\sqrt{\epsilon},\sqrt{\epsilon}]\times[-\sqrt{\epsilon},\sqrt{\epsilon}].
\end{align*}

This imply that $\text{diam}(S(\epsilon))\leq 2\sqrt{2\epsilon}\rightarrow 0$ as $\epsilon\rightarrow 0$. By Theorem \ref{Thm:1}, 
the SEP (\ref{SEP}) is well-posed.
\end{example}

\begin{theorem}\label{Thm:2}
Let  $f:C\times C\rightarrow\mathbb{R}$ and $g:Q\times Q\rightarrow\mathbb{R}$ be two upper semi-continuous bifunctions
in the first variable. Then the SEP \eqref{SEP} is well-posed if and only if
\begin{align}\label{in2}
S(\epsilon)\neq\emptyset \  (\forall \epsilon>0) \quad\text{and} \quad
\lim_{\epsilon\rightarrow 0}\text{diam}(S(\epsilon))=0.
\end{align}
\end{theorem}

\begin{proof}

From Theorem \ref{Thm:1}, it is clear that we need to prove only the sufficiency part.
Suppose that the condition (\ref{in2}) holds. Since $S\subset S(\epsilon)$ for each $\epsilon>0$, the SEP (\ref{SEP}) admits at most one solution.

Let $(x_n, y_n)$ be an approximate sequence for the SEP (\ref{SEP}). Then there exists a positive sequence $\{\epsilon_n\}$ such that $\epsilon_n\rightarrow 0$ and 
\begin{equation}\label{ASEP}
    \begin{cases}
      x_n\in C, \quad y_n\in Q,\\
      \|y_n-Ax_n\|\leq \epsilon_n ,\\
      f(x_n, p)\geq -\epsilon_n, \quad\forall p\in C,\\
      g(y_n, q)\geq -\epsilon_n, \quad\forall q\in Q.
    \end{cases}
\end{equation}

It is evident that $(x_n, y_n)\in S(\epsilon_n)$. Now
the condition $\lim_{n\to\infty}\text{diam}(S(\epsilon_n))=0$ ensures that $\{(x_n, y_n)\}$ is Cauchy.
Let $(x_n, y_n)\rightarrow (x^*, y^*)$. Since $f$ and $g$ are upper semi-continuous in the first variable,
by taking the limit as $n\to\infty$ in \eqref{ASEP}, we immediately arrive at
\begin{equation*}%\label{ASEP2}
    \begin{cases}
      x^*\in C, \quad y^*\in Q,\\
      y^*=Ax^*, \\
      f(x^*, p)\geq 0, \quad\forall p\in C,\\
      g(y^*, q)\geq 0, \quad\forall q\in Q.
    \end{cases}
\end{equation*}
Thus $(x^*, y^*)$ is the unique solution to SEP \eqref{SEP}. This completes the proof.
\end{proof}

\begin{remark}
The above Theorem \ref{Thm:2} not only provides the well-posedness of the SEP (\ref{SEP}), but also its equivalence with the existence and uniqueness
of its solution.
\end{remark}

The following characterization of generalized well-posedness via Hausdorff metric provides the connection between the approximate solution set and the solution set of SEP. The concept of generalized well-posedness allows us to remove the unnecessary uniqueness of the solution to SEP.

\begin{theorem}\label{Thm:3}
The SEP \eqref{SEP} is generalized well-posed if and only if the solution set $S$ is nonempty compact, and
\begin{align}\label{in3}
H(S(\epsilon), S)\rightarrow 0~\text{as}~\epsilon\rightarrow 0.
\end{align}
\end{theorem}

\begin{proof}
Suppose that SEP \eqref{SEP} is generalized well-posed. It is clear that $\emptyset\neq S\subset S(\epsilon)$ for all $\epsilon>0$. Then each
sequence $\{(x_n,y_n)\}$ in $S$ is evidently an approximating sequence for the SEP \eqref{SEP}.
Hence, the generalized well-posedness straightforwardly implies that a subsequence of $\{(x_n,y_n)\}$
converges to a point in $S$. This verifies the compactness of $S$. We now turn to prove \eqref{in3}.
Since $S(\epsilon)\supset S$, we have
\begin{align*}
H(S(\epsilon),S)=\sup\{d(z,S): z\in S(\epsilon)\}
\end{align*}
is decreasing in $\epsilon\ge 0$. Let $h:=\lim_{\epsilon\to 0}H(S(\epsilon),S)$. Suppose $h>0$.
Then, for $0<\epsilon_n$ decreasing to zero, we can find $(x_n,y_n)\in S(\epsilon_n)$ such that
\begin{equation}\label{eq:d1}
d((x_n,y_n),S)>\frac12 h, \quad n\ge 1.
\end{equation}
However, the generalized well-posedness yields a subsequence $\{(x_{n_k},y_{n_k})\}$ of $\{(x_n,y_n)\}$
converging to some point $(\hat{x},\hat{y})\in S$, which results from \eqref{eq:d1} in that $0\ge h$.
This contradicts to the assumption $h>0$. Hence, we must have $h=0$.

Conversely, let us assume \eqref{in3} with $S$ being nonempty compact.
To show the generalized well-posedness of the SEP \eqref{SEP}, let $\left\lbrace (x_n, y_n)\right\rbrace$
be an approximating sequence for the SEP \eqref{SEP}.
Then there exists $0<\epsilon_n\rightarrow 0$ such that $(x_n, y_n)\in S(\epsilon_n).$
\begin{align*}
d((x_n, y_n), S)\leq D(S(\epsilon_n), S)\le H(S(\epsilon_n), S)\rightarrow 0.
\end{align*}
Since $S$ is compact, $\left\lbrace (x_n, y_n)\right\rbrace$ has a subsequence converging to some element of $S$.
Hence, the generalized well-posedness of the SEP \eqref{SEP} is proven.
\end{proof}

\begin{example}
Let $H_1=H_2=\mathbb{R}$, $A=I$ and $C=Q=[0, 1]$. Let us define $f:C\times C\rightarrow\mathbb{R}$ by
\begin{equation*}
f(x, p)=
\begin{cases}
x, \quad x\in [0, \frac{1}{2})\\
\frac{x^2}{2}, \quad[\frac{1}{2}, 1].
\end{cases}
\end{equation*}
and $g:Q\times Q\rightarrow\mathbb{R}$ by
\begin{equation*}
g(y, q)=
\begin{cases}
0, \quad y=\frac{1}{2}\\
2, \quad \text{otherwise}.
\end{cases}
\end{equation*}

It is easy to check that $S=\{(x^*,y^*)\in [0,1]\times[0, 1]:y^*=x^*\}$ is nonempty compact. Also, one can verify that $H(S(\epsilon), S)\rightarrow 0$ as $\epsilon\rightarrow 0$.
Therefore, by Theorem \ref{Thm:3}, the SEP (\ref{SEP}) is generalized well-posed (but not well-posed).

\end{example}

We establish the following theorem by assuming upper semi-continuity of $f$ and $g$ and weakening the condition on the solution set $S$
via the Kuratowski measure of noncompactness.

\begin{theorem}\label{Thm:4}
Let $f: C\times C\rightarrow\mathbb{R}$ and $g: Q\times Q\rightarrow\mathbb{R}$ be two upper semi-continuous bifunctions in the first variable.
Then the SEP \eqref{SEP} is generalized well-posed if and only if
\begin{align}\label{in5}
S(\epsilon)\neq\emptyset\quad (\forall\epsilon>0) \quad\text{and} \quad \alpha(S(\epsilon))\rightarrow 0\quad\text{as} \quad\epsilon\rightarrow 0,
\end{align}
where $\alpha$ is Kuratowski's measure of noncompactness.
\end{theorem}

\begin{proof}
Suppose that the SEP \eqref{SEP} is generalized well-posed. Thus $S(\epsilon)\supset S\neq\emptyset$ for all $\epsilon>0$.
By Theorem \ref{Thm:3}, we know that $S$ is compact (thus $\alpha(S)=0$) and
\begin{align}\label{in6}
H(S(\epsilon), S)\rightarrow 0~\text{as}~\epsilon\rightarrow 0.
\end{align}
It follows that
\begin{align*}
\alpha(S(\epsilon))\leq 2H(S(\epsilon), S)+\alpha(S)=2H(S(\epsilon), S).
\end{align*}
This together with $(\ref{in6})$ implies that
$\alpha(S(\epsilon))\rightarrow 0$ as $\epsilon\rightarrow 0.$

Conversely, let us assume \eqref{in5} and prove that the SEP \eqref{SEP} is generalized well-posed.
Since $f$ and $g$ are both upper semi-continuous in the first variable, the set $S(\epsilon)$ is closed for all $\epsilon>0.$ It is easy to verify that
\begin{align}
S=\cap_{\epsilon>0}S(\epsilon).\nonumber
\end{align}
Since
$\alpha(S(\epsilon))\rightarrow 0$, by Theorem 1 \cite{CHOR1985} (or, one can see, \cite[pp. 412]{KKUR1968}),  $S$ is nonempty compact and
\begin{align*}
H(S(\epsilon), S)\rightarrow 0~\text{as}~\epsilon\rightarrow 0.\nonumber
\end{align*}
Therefore, by Theorem \ref{Thm:3}, the SEP \eqref{SEP} is generalized well-posed.
\end{proof}

\begin{remark}
The above Theorem \ref{Thm:4} shows that the generalized well-posedness of the problem (\ref{SEP}) is related to the compactness of the approximate solution set.
\end{remark}

%====================================================================================================
\section{Well-Posedness and Uniqueness of Solution}
\label{Sec:4}

The most interesting problems in the theory of well-posedness for variational problems is to draw the equivalence between the well-posedness and uniqueness of the solution. In this section, we shall prove that the well-posedness of the SEP (\ref{SEP}) is equivalent to the existence and uniqueness of its solution under mild conditions.

\begin{theorem}\label{Thm:5}
Let $C$ and $Q$ be nonempty closed convex subsets of finite-dimensional real Hilbert spaces $H_1$ and $H_2$, respectively.
Let $f:C\times C\rightarrow\mathbb{R}$ and $g:Q\times Q\rightarrow\mathbb{R}$ be hemicontinuous bifunctions which satisfy the following conditions:
\begin{itemize}
\item[(i)]  $f$ and $g$ are both monotone.
\item[(ii)] $f$ and $g$ are both convex in the second variable.
\item[(iii)] $f$ and $g$ are both lower semi-continuous in the second variable.
\item[(iv)] $f(p, p)$ and $g(q, q)$ are nonnegative for every $(p, q)\in C\times Q.$
\end{itemize}
Then the SEP \eqref{SEP} is well-posed if and only if the SEP \eqref{SEP} has a unique solution.
\end{theorem}

\begin{proof}
Since the necessary part of the above theorem is trivial, let us prove only its sufficient part.
Suppose that $(x^*, y^*)$ is the unique solution to the SEP \eqref{SEP}. Then we have the following
\begin{equation}\label{in7}
\begin{cases}
    x^*\in C, \quad y^*=Ax^*\in Q,\\
    f(x^*, p)\geq 0, \quad\forall p\in C,\\
    g(y^*, q)\geq 0, \quad\forall q\in Q.
\end{cases}
\end{equation}
Since $f$ and $g$ are monotone, we get from $(\ref{in7})$ that
\begin{equation}\label{in8}
\begin{cases}
    f(p, x^*)\leq-f(x^*, p)\leq 0, \quad\forall p\in C,\\
    g(q, y^*)\leq-g(y^*, q)\leq 0, \quad\forall q\in Q.
\end{cases}
\end{equation}

Let $\left\lbrace(x_n, y_n)\right\rbrace$ be an approximating sequence of the SEP \eqref{SEP}. Then there exists
$0<\epsilon_n\rightarrow 0$ such that
\begin{equation}\label{ASEP4}
\begin{cases}
      x_n\in C, \quad y_n\in Q,\\
      \|y_n-Ax_n\|\leq \epsilon_n ,\\
      f(x_n, p)\geq -\epsilon_n, \quad\forall p\in C,\\
      g(y_n, q)\geq -\epsilon_n, \quad\forall q\in Q.
\end{cases}
\end{equation}
Therefore, using the monotonicity of $f$ and $g$ from $(\ref{ASEP4})$, we have the following
\begin{equation}\label{ASEP5}
\begin{cases}
      x_n\in C, \quad y_n\in Q,\\
      \|y_n-Ax_n\|\leq \epsilon_n, \\
      f(p, x_n)\leq -f(x_n, p)\leq \epsilon_n, \quad\forall p\in C,\\
      g(q, y_n)\leq- g(y_n, q)\leq \epsilon_n, \quad\forall q\in Q.
\end{cases}
\end{equation}

Now we have to show that $\left\lbrace(x_n, y_n)\right\rbrace$ is bounded. Suppose that $\left\lbrace(x_n, y_n)\right\rbrace$ is not bounded. With no loss of generality, we can assume that $\|(x_n, y_n)\|\rightarrow \infty.$ Set $u^*=(x^*, y^*)$, $u_n=(x_n, y_n)$ and
\begin{align*}
t_n&=\frac{1}{\|(x_n, y_n)-(x^*, y^*)\|}=\frac{1}{\|u_n-u^*\|},\\
v_n&=(z_n, w_n)=u^*+t_n(u_n-u^*)=((1-t_n)x^*+t_n x_n, (1-t_n)y^*+t_n y_n).
\end{align*}

We may assume that $t_n\in(0, 1]$ and $v_n\rightarrow\bar{v}=(\bar{z}, \bar{w})\neq u^*$ because $H_1$ and $H_2$ are finite-dimensional.

Since $f$ is convex and lower semi-continuous in the second variable, we get
\begin{align}\label{in9}
f(p, \bar{z})&\leq\liminf_{n\rightarrow\infty} f(p, z_n)\\&
\leq\liminf_{n\rightarrow\infty} f(p, (1-t_n)x^*+t_n x_n)\nonumber\\&
\leq \liminf_{n\rightarrow\infty}[(1-t_n)f(p, x^*)+ t_n f(p, x_n)]\nonumber\\&
\leq\liminf_{n\rightarrow\infty}~t_n\epsilon_n=0,\quad \forall p\in C.\nonumber
\end{align}
Similarly, since $g$ is convex and lower semi-continuous in the second variable, we have
\begin{align}\label{in9}
g(q, \bar{w})&\leq\liminf_{n\rightarrow\infty} g(q, w_n)\\&
\leq\liminf_{n\rightarrow\infty} g(q, (1-t_n) y^*+t_n y_n)\nonumber\\&
\leq \liminf_{n\rightarrow\infty}[(1-t_n)g(q, y^*)+ t_n g(q, y_n)]\nonumber\\&
\leq \liminf_{n\rightarrow\infty}~t_n\epsilon_n=0,\quad\forall q\in Q.\nonumber
\end{align}
Further, we also have
\begin{align}
\|A(z_n)-w_n\|&=\|A((1-t_n)x^*+t_n x_n)-(1-t_n)y^*-t_n y_n\|\\
&=t_n\|Ax_n-y_n\|\leq t_n\epsilon_n\rightarrow 0\nonumber.
\end{align}

Since $z_n\rightarrow\bar{z}$ and $w_n\rightarrow\bar{w}$, $A(\bar{z})=\bar{w}$. Using Theorem \ref{Lem:1} and uniqueness of the solution to SEP it is easy to show that $(\bar{z}, \bar{w})=(x^*, y^*)=u^*,$ which is a contradiction.
Thus $\left\lbrace u_n\right\rbrace$ is bounded.

Let $\left\lbrace u_{n_k}\right\rbrace$ be any subsequence of $\left\lbrace u_n\right\rbrace$. Since $\left\lbrace u_{n}\right\rbrace$ is bounded, $\left\lbrace u_{n_k}\right\rbrace$ has a convergent subsequence $\left\lbrace u_{n_{k_l}}\right\rbrace$ with limit $\bar{u}=(\bar{x}, \bar{y}).$ By our assumptions, it is easy to show that $\bar{u}=(\bar{x}, \bar{y})$ is a solution to the SEP (\ref{SEP}). Since $(x^*, y^*)=u^*$ is the unique solution to the SEP (\ref{SEP}), we get $\bar{u}=u^*$ and thus the approximating sequence $\left\lbrace(x_n, y_n)\right\rbrace$ converges to the unique solution of the SEP (\ref{SEP}). Therefore, the SEP (\ref{SEP}) is well-posed.
\end{proof}

%=====================================================================================================
The following example illustrates our Theorem $\ref{Thm:5}$.
\begin{example}
Let $H_1=H_2=\mathbb{R}$, $A=I$, $C=Q=[0, +\infty)$, $f(x, p)=p^2-x^2$ and $g(y, q)=q-y$. 
It is easy to check that the bifunctions $f$ and $g$  satisfy all the conditions of Theorem $\ref{Thm:5}$. 
Also, one can show that the point $(0, 0)$ is a unique solution of the SEP (\ref{SEP}).  
Therefore, by Theorem $\ref{Thm:5}$, the SEP (\ref{SEP}) is well-posed.
\end{example}

%=======================================================================================================

%=======================================================================================================
\section{Conclusion}\label{Sec:5}
In this paper we introduced the concept of well-posedness to the split equilibrium problem in infinite-dimensional real Hilbert spaces and Furi-Vignoli type of characterizations have been established for the well-posedness. We also showed that the well-posedness of the SEP (\ref{SEP}) is equivalent to the existence and uniqueness of its solution in finite-dimensional real Hilbert spaces. In addition, we provided some nontrivial examples to illustrate our theoretical results. It is interesting to ask the following question: 

\noindent Does Theorem \ref{Thm:5} hold in the infinite-dimensional Hilbert spaces?

\noindent The answer to this question remains open.
  
%=====================================================================================

\section{Acknowledgments} 
The authors are thankful to the editor and to two anonymous referees for their useful comments and helpful suggestions.

\section{ Declarations}

\subsection {\bf Conflicts of interest}

There is no conflict of interest.
 
\subsection{\bf  Data availability statement}

The authors acknowledge that the data presented in this study must be deposited and made publicly
available in an acceptable repository, prior to publication.

%====================================================================================================

% Non-BibTeX users please use

\end{document}